\documentclass[english]{amsart}
\usepackage[T1]{fontenc}
\usepackage[latin9]{inputenc}
\usepackage{geometry}
\usepackage{mathrsfs}
\usepackage{amsthm}
\usepackage{amssymb}

\makeatletter

\newcommand{\lyxdeleted}[3]{}

\numberwithin{equation}{section}
\numberwithin{figure}{section}
\theoremstyle{plain}
\newtheorem{thm}{\protect\theoremname}
  \theoremstyle{plain}
  \newtheorem{lem}[thm]{\protect\lemmaname}
  \theoremstyle{plain}
  \newtheorem{prop}[thm]{\protect\propositionname}
  \theoremstyle{remark}
  \newtheorem*{claim*}{\protect\claimname}

\usepackage[all]{xy}

\makeatother

\usepackage{babel}
  \providecommand{\claimname}{Claim}
  \providecommand{\lemmaname}{Lemma}
  \providecommand{\propositionname}{Proposition}
\providecommand{\theoremname}{Theorem}

\begin{document}

\title[Criteria for periodicity]{Criteria for periodicity and an application to elliptic functions}

\author{Ehud de Shalit}
\begin{abstract}
Let $P$ and $Q$ be relatively prime integers greater than 1, and
$f$ a real valued discretely supported function on a finite dimensional
real vector space $V$. We prove that if $f_{P}(x)=f(Px)-f(x)$ and
$f_{Q}(x)=f(Qx)-f(x)$ are both $\Lambda$-periodic for some lattice
$\Lambda\subset V$, then so is $f$ (up to a modification at $0$).
This result is used to prove a theorem on the arithmetic of elliptic
function fields. In the last section we discuss the higher rank analogue
of this theorem and explain why it fails in rank 2. A full discussion
of the higher rank case will appear in a forthcoming work.
\end{abstract}

\address{Einstein Institute of Mathematics, Hebrew University of Jerusalem}

\email{ehud.deshalit@mail.huji.ac.il}

\date{January 23, 2020}

\thanks{The author was supported by ISF grant 276/17.}
\maketitle

\section*{Introduction}

Let $V$ be an $r$-dimensional vector space over $\mathbb{R}$ and
$\mathscr{D}$ the abelian group of \emph{discretely supported} functions\footnote{We call $f$ discretely supported if $\{x\in V|f(x)\ne0\}$ has no
accumulation points in $V$.} $f:V\to\mathbb{R}$. If $P\ge2$ is an integer and $f\in\mathscr{D}$
we let
\[
f_{P}(x)=f(Px)-f(x)\in\mathscr{D}.
\]
Note that $f_{P}$ is insensitive to the value of $f$ at $0,$ namely
we may modify $f$ at 0 without affecting $f_{P}$. We henceforth
call $f'$ a \emph{modification} \emph{of} $f$ \emph{at} $0$ if
$f'(x)=f(x)$ at every $x\ne0.$

Let $\Lambda\subset V$ be a lattice. Our interest lies in the subgroup
$\mathscr{P}$ of $f\in\mathscr{D}$ satisfying the periodicity condition
\[
f(x+\lambda)=f(x)\,\,\,(\forall\lambda\in\Lambda).
\]
If $f\in\mathscr{P}$ then clearly $f_{P}(0)=0$ and $f_{P}\in\mathscr{P}$.
The converse is false, even if we allow to modify $f$ at $0$. Indeed,
let $V=\mathbb{R}$, $\Lambda=\mathbb{Z}.$ Let $f_{P}$ be any non-zero
$\mathbb{Z}$-periodic function vanishing at $0$ and
\[
f(x)=\sum_{i=1}^{\infty}f_{P}(x/P^{i}).
\]
Observe that for every $x$ the sum is finite, and that $f\in\mathscr{D}.$
Then $f(Px)-f(x)=f_{P}(x),$ but $f$ need not be periodic. If $f_{P}\ge0$
and is supported on non-integral rational numbers whose denominators
are relatively prime to $P$, then $f$ is even unbounded.

In the first part of this note we prove the following theorem.
\begin{thm}
\label{Periodicity Theorem}Let $P$ and $Q$ be greater than 1 and
relatively prime integers. If both $f_{P}$ and $f_{Q}$ are $\Lambda$-periodic,
so is a suitable modification of $f$ at $0$.
\end{thm}

The proof is elementary, but somewhat tricky. It is possible that
the theorem remains valid if $P$ and $Q$ are only multiplicatively
independent ($P^{a}=Q^{b}$ for $a,b\in\mathbb{Z}$ if and only if
$a=b=0$). Our methods do not yield this generalization, although
we do obtain a partial result along the way, see Proposition \ref{prop:non-torsion}.

Taking $V=\mathbb{R},$ $\Lambda=\mathbb{Z}$ and $f(x)=1$ if $0\ne x\in\mathbb{Z}$
and $0$ elsewhere, we get that $f_{p}$ is $\mathbb{Z}$-periodic
for any prime $p$. This shows that we can not forgo the modification
at 0, even if we replace it by the condition $f(0)=0.$

\bigskip{}

In the second part of our note we derive from Theorem \ref{Periodicity Theorem}
a theorem on elliptic functions. Here we take, of course, $V=\mathbb{C}.$
The relation with elliptic functions comes from the fact that the
\emph{divisor function $e=div(f),$} (i.e. $e(z)=ord_{z}(f)$) of
a $\Lambda$-elliptic function $f$ lives in $\mathscr{P}$, and determines
$f$ up to a multiplicative constant. We refer to the text for the
precise formulation of our main result, see Theorem \ref{thm:d=00003D1_case_conjecture}.
Besides Theorem \ref{Periodicity Theorem}, its proof uses only basic
facts on elliptic functions (the Abel-Jacobi theorem). Here we mention
an immediate corollary.
\begin{thm}
\label{thm:criterion for ellipticity}Let $P$ and $Q$ be greater
than 1 and relatively prime integers. Let $f$ be a meromorphic function
on $\mathbb{C}$ for which $f_{P}(z)=f(Pz)/f(z)$ and $f_{Q}(z)=f(Qz)/f(z)$
are $\Lambda$-elliptic. Then there exists a lattice $\Lambda'\subset\Lambda$
and an integer $m$ such that $z^{m}f(z)$ is $\Lambda'$-elliptic.
If $\gcd(P-1,Q-1)=D$ we can take $\Lambda'=D\Lambda.$
\end{thm}

\bigskip{}

In the third and last section we discuss our motivation: an elliptic
analogue of \emph{a conjecture of Loxton and van der Poorten}, proved
by Adamczewski and Bell in \cite{2}. Again we refer to the text for
details. The original proof of this celebrated conjecture relied on
Cobham's theorem in the theory of automata, whose proof in \cite{3}
was notoriously long and complicated. Recently, Schäfke and Singer
found an independent proof \cite{7} that both clarified the ideas
involved, and eliminated the dependence on Cobham's theorem. In fact,
as was known to the experts, the latter \emph{follows} in turn from
the Loxton-van der Poorten conjecture, so \cite{7} yields a conceptual
and relatively short proof of Cobham's theorem as an added bonus.
For more on this circle of ideas and related work, see the survey
paper by Adamczewski \cite{1}.

Although not explicitly stated so in \cite{7}, the mechanism behind
the proof of Schäfke and Singer is cohomological. Reformulating their
work \cite{4} lead us to a similar question in the elliptic set-up,
involving a certain non-abelian cohomology of $\Gamma\simeq\mathbb{Z}^{2}$
with coefficients in $GL_{d}(K)$, where $K$ is the maximal unramified
extension of the field of $\Lambda$-elliptic functions. While theorem
\ref{thm:criterion for ellipticity} amounts to a positive answer
to the case $d=1$ of this question, we give an example showing that
already for $d=2$ the answer is negative.

The complete solution of the question raised in the last part amounts
to a classification of objects that we call, in a forthcoming paper
\cite{8}, \emph{elliptic $(P,Q)$-difference modules. }In that work
we show how a generalization of the periodicity criterion of Theorem
1 leads to a connection between this classification problem and the
classification of vector bundles on elliptic curves, a result of Atiyah
from 1957. For $d=2$ this suffices to complete the classifictaion
of rank-2 elliptic $(P,Q)$-difference modules and deduce that, ``up
to a twist'', our counter-example is the only such counter-example.
We hope to settle the higher rank question completely in \cite{8}.

\section{The theorem on periodic functions}

\subsection{A Lemma}

We begin with an elementary lemma. Fix an integer $N\ge1$. If $0\ne x\in\mathbb{Z}$
and $p$ is a prime number we write $v_{p}(x)$ for the power of $p$
dividing $x$. If $S$ is a set of primes we write
\[
x'_{S}=\prod_{p\in S}p^{-v_{p}(x)}\cdot x,
\]
for the ``prime-to-$S$'' part of $x$ (retaining the sign).

For non-zero $x,y\in\mathbb{Z}$ we define $x\sim_{S}y$ to mean $v_{p}(x)=v_{p}(y)$
for every $p\in S$ and $x'_{S}\equiv y'_{S}\mod N$. This is clearly
an equivalence relation on $\mathbb{Z}$ (where, by convention, the
equivalence class of $0$ is a singleton). For example, when $N=10$
and $S=\{5\},$ $12\sim_{S}32$ and $15\sim_{S}65$ but $15\nsim_{S}35.$
\begin{lem}
\label{lem:pq_Lemma}Let $S$ and $T$ be disjoint, non-empty finite
sets of primes and $\sim$ the equivalence relation on $\mathbb{Z}$
generated by $\sim_{S}$ and $\sim_{T},$ namely $x\sim y$ if there
exists a sequence $x=x^{(1)},\dots,x^{(K)}=y$ such that for every
$i$, $x^{(i)}\sim_{S}x^{(i+1)}$ or $x^{(i)}\sim_{T}x^{(i+1)}.$
Assume that $x,y\ne0.$ Then $x\sim y$ if and only if $x\equiv y\mod N$.
\end{lem}

\begin{proof}
Let $m_{p}=v_{p}(x)+1$ ($p\in S$) and $n_{q}=v_{q}(y)+1$ ($q\in T).$
Let
\[
P=\prod_{p\in S}p^{m_{p}},\,\,\,Q=\prod_{q\in T}q^{n_{q}}.
\]
Assume that $y=x+kN$ and let $s$ and $t$ satisfy
\[
sP-tQ=k.
\]
Then
\[
z=x+sPN=y+tQN
\]
and it is easily checked that $x\sim_{S}z$ and $z\sim_{T}y.$ Thus
$x\sim y.$ The converse is obvious
\end{proof}

\subsection{A Proposition}

We use the same notation as in the introduction. In particular $V$
is a real $r$-dimensional vector space, and $\Lambda$ is a lattice
in $V$.
\begin{prop}
\label{prop:rational_case}Let $P$ and $Q$ be greater than 1 and
relatively prime integers. Let $f\in\mathscr{D}$ be a function supported
on $PQ\Lambda$. Let
\begin{equation}
f_{P}(x)=f(Px)-f(x),\,\,\,f_{Q}(x)=f(Qx)-f(x).\label{eq:f_p_def}
\end{equation}
If both $f_{P}$ and $f_{Q}$ are $N\Lambda$-periodic then a certain
modification of $f$ at 0 is $N\Lambda$-periodic.
\end{prop}

\begin{proof}
Observe first that $f_{P}$ is supported on $Q\Lambda$ and $f_{Q}$
is supported on $P\Lambda$, so $N$ is divisible by $PQ.$ For every
$0\ne x\in V$ equations (\ref{eq:f_p_def}) give the relations
\begin{equation}
f(x)=\sum_{i=1}^{\infty}f_{P}(x/P^{i})=\sum_{j=1}^{\infty}f_{Q}(x/Q^{j}),\label{eq:telescoping}
\end{equation}
both sums being finite. Fix $0\ne x,y\in\Lambda$ such that $x-y\in N\Lambda$.
We shall show that $f(x)=f(y)$. In particular there will be a constant
$c$ such that $f(x)=c$ for every $0\ne x\in N\Lambda$. Modifying
$f$ to obtain the value $c$ at 0 too, we get an $N\Lambda$-periodic
function.

Fix a basis of $\Lambda$ over $\mathbb{Z}$ in which the coordinates
of $x$ and $y$ are all non-zero. This is always possible, and we
call such a basis \emph{adapted} to $x$ and $y$. Using this basis
we identify $\Lambda$ with $\mathbb{Z}^{r}$ and $V$ with $\mathbb{R}^{r}$.
Instead of congruences modulo $N\Lambda$ we write congruences modulo
$N$.

Let $S$ be the set of primes dividing $P$ and $T$ the set of primes
dividing $Q$. For $u$ and $v$ in $\mathbb{Z}^{r}$ write $u\sim_{S}v$
if this equivalence relation holds coordinate-wise. In particular,
if the $\nu$-th coordinate of $u$ vanishes, so must the $\nu$-th
coordinate of $v$.

Since $x\equiv y\mod N$ and none of the coordinates of $x$ or $y$
vanishes, there is a sequence
\[
x=x^{(1)},\dots,x^{(K)}=y
\]
of vectors in $\mathbb{Z}^{r}$ such that for each $l$ we have $x^{(l)}\sim_{S}x^{(l+1)}$
or $x^{(l)}\sim_{T}x^{(l+1)}.$ (In fact the proof of Lemma \ref{lem:pq_Lemma}
shows that we can take $K=3.$) It is therefore enough to show that
if $x\sim_{S}y$ then $f(x)=f(y).$ Assume therefore that $x\sim_{S}y.$

Write $x=P^{m}x'$ and $y=P^{m}y'$ where $x'$ and $y'$ are in $\mathbb{Z}^{r}$
but not in $P\mathbb{Z}^{r}$. That the same $m$ works for both $x$
and $y$ follows from the fact that for each $1\le\nu\le r,$ the
$p$-adic valuations of the $\nu$-th coordinates $v_{p}(x_{\nu})=v_{p}(y_{\nu})$
for every prime $p|P.$ Since $f_{P}$ is supported on $\mathbb{Z}^{r}$,
equation (\ref{eq:telescoping}) implies
\[
f(x)=\sum_{i=0}^{m-1}f_{P}(P^{i}x').
\]
But $x\sim_{S}y$ implies that $P^{i}x'\equiv P^{i}y'\mod N$ . Since
$f_{P}$ is $N$-periodic we get that
\[
f(x)=\sum_{i=0}^{m-1}f_{P}(P^{i}y')=f(y).
\]
This concludes the proof of the Proposition.
\end{proof}

\subsection{The proof of Theorem \ref{Periodicity Theorem}}

Let $f\in\mathscr{D}$ be as in the Theorem, $P,Q\ge2$. Let $\Lambda$
be a lattice of periodicity for $f_{P}$ and $f_{Q}$$.$ Our goal
is to show that if $(P,Q)=1$ the function $f$, appropriately modified
at $0$, is also $\Lambda$-periodic.

Denote by $S_{P},S_{Q}\subset V/\Lambda$ the supports of $f_{P}$
and $f_{Q}$ and by $\widetilde{S}_{P}$ and $\widetilde{S}_{Q}$
their pre-images in $V$$.$ Let $\widetilde{S}$ be the support of
$f$.
\begin{lem}
\label{lem:supports}Assume that $P$ and $Q$ are multiplicatively
independent. Then the projection $\widetilde{S}\mod\Lambda$ is finite.
\end{lem}

\begin{proof}
Equation (\ref{eq:telescoping}) holds for every $x\in V$ and shows
that $\widetilde{S}$ is contained in
\[
\bigcup_{n=1}^{\infty}P^{n}\widetilde{S}_{P}\cap\bigcup_{m=1}^{\infty}Q^{m}\widetilde{S}_{Q}.
\]
It is therefore enough to prove that $\bigcup_{n=1}^{\infty}P^{n}S_{P}\cap\bigcup_{m=1}^{\infty}Q^{m}S_{Q}$
is finite. The sets $S_{P}$ and $S_{Q}$ are of course finite. Let
$\bar{z}=z\mod\Lambda\in S_{P}$ and $\bar{w}=w\mod\Lambda\in S_{Q}$,
$n$ and $m$ be such that $P^{n}\bar{z}=Q^{m}\bar{w}.$ If $z$ (hence
also $w$) lies in $M=\mathbb{Q}\Lambda$ then there are altogether
only finitely many points of the form $P^{n}\bar{z}$ in $V/\Lambda$.
It is therefore enough to assume that $z,w\notin M$ and prove that
$(n,m)$ are then uniquely determined by $(z,w).$ But suppose $P^{n}z\equiv Q^{m}w\mod\Lambda$
and also $P^{n'}z\equiv Q^{m'}w\mod\Lambda,$ where without loss of
generality we may assume $n'>n.$ Then
\[
(P^{n'-n}Q^{m}-Q^{m'})w\in\Lambda,
\]
contradicting the assumption that $w\notin M$. In the last step we
used the multiplicative independence of $P$ and $Q$ to guarantee
that the coefficient of $w$ is non-zero.
\end{proof}
We continue with the proof, assuming only that $P$ and $Q$ are multiplicatively
independent. Let $S$ be the projection of $\widetilde{S}$ modulo
$\Lambda$. Pick $z\in\widetilde{S}_{P},$ $z\notin M=\mathbb{Q}\Lambda$.
We call $\{z,Pz,P^{2}z,...\}\cap\widetilde{S}_{P}$ the $P$\emph{-chain
through} $z$. Since $z\notin M$ all the $P^{n}z$ have distinct
images modulo $\Lambda$, so only finitely many of them belong to
$\widetilde{S}_{P}.$ Let $P^{n(z)}z$ be the last one, and call $n(z)\ge0$
the \emph{exponent} of the $P$-chain through $z$. Call a $P$-chain
\emph{primitive }if it is not properly contained in any other $P$-chain,
i.e. if none of the points $P^{n}z,$ $n<0,$ belongs to $\widetilde{S}_{P}$.
Since $\widetilde{S}_{P}$ is $\Lambda$-periodic, $n(z+\lambda)=n(z)$
for $\lambda\in\Lambda.$ It follows from the discreteness of $\widetilde{S}_{P}$
that
\[
n_{P}=1+\max_{z\in\widetilde{S}_{P},\,z\notin M}n(z)<\infty.
\]

Let $\{z,Pz,\dots,P^{n(z)}z\}\cap\widetilde{S}_{P}$ be a primitive
$P$-chain through $z\notin M.$ We claim that
\begin{equation}
\sum_{i=0}^{n(z)}f_{P}(P^{i}z)=0.\label{eq:chain vanishing}
\end{equation}
Indeed, for every $n>n(z)$
\[
f(P^{n}z)=\sum_{i=1}^{\infty}f_{P}(P^{n-i}z)=\sum_{i=0}^{n(z)}f_{P}(P^{i}z)
\]
so the assertion follows from Lemma \ref{lem:supports}, since otherwise
all $P^{n}z,$ $n>n(z)$, would lie in $\widetilde{S}$, and they
are all distinct modulo $\Lambda$. It follows also that $f(P^{n}z)=0$
if $n<0$ or $n>n(z).$

\bigskip{}

Let $\lambda\in\Lambda.$ Assume $z\notin M$ and $f(z)\ne0.$ Then
\[
f(z)=\sum_{i=1}^{n_{P}}f_{P}(P^{-i}z).
\]
The reason we can stop at $i=n_{P}$ is that if $i_{0}$ is the largest
index such that $f_{P}(P^{-i}z)\ne0$ and $i_{0}>n_{P}$ then $f(z)=\sum_{i=1}^{\infty}f_{P}(P^{-i}z)=0$
by (\ref{eq:chain vanishing}) applied to $P^{-i_{0}}z$ instead of
$z$. Thus if $f(z)\ne0$ we must have $i_{0}\le n_{P}$. By the periodicity
of $f_{P}$ we now have
\[
f(z)=\sum_{i=1}^{n_{P}}f_{P}(P^{-i}(z+P^{2n_{P}}\lambda)).
\]
The last sum is equal to $\sum_{i=1}^{2n_{P}}f_{P}(P^{-i}(z+P^{2n_{P}}\lambda))$
because the terms with $n_{P}<i\le2n_{P}$ all vanish as they are
equal to $f(P^{-i}z),$ which, as we have just seen, vanish. Since
one of the terms $f_{P}(P^{-i}(z+P^{2n_{P}}\lambda))$ with $i\le n_{P}$
must not vanish, and the exponent of any primitive $P$-chain is less
than $n_{P},$ the terms $f_{P}(P^{-i}(z+P^{2n_{P}}\lambda))$ with
$i>2n_{P}$ all vanish. We conclude that
\[
f(z)=\sum_{i=1}^{\infty}f_{P}(P^{-i}(z+P^{2n_{P}}\lambda))=f(z+P^{2n_{P}}\lambda).
\]
To sum up, we have shown that if $z\notin M$ and $f(z)\ne0$ then
$f(z)=f(z+P^{2n_{P}}\lambda)$ for every $\lambda\in\Lambda$. This
of course stays true if $f(z)=0$, for if $f(z+P^{2n_{P}}\lambda)\ne0$
switch the roles of $z$ and $z+P^{2n_{P}}\lambda$ and replace $\lambda$
by $-\lambda$.

Repeating the same arguments with $Q$ replacing $P$ we get that
\[
f(z)=f(z+q^{2n_{Q}}\lambda)
\]
for all $z\notin M.$ If $\gcd(P,Q)=1,$ the lattice generated by
$P^{2n_{P}}\Lambda$ and $Q^{2n_{Q}}\Lambda$ is $\Lambda$. We therefore
get the following conclusion:
\begin{prop}
\label{prop:non-torsion}Let $f\in\mathscr{D}$ and assume that $P$
and $Q$ are multiplicatively independent. If $f_{P}$ and $f_{Q}$
are $\Lambda$-periodic then there exists a lattice $\Lambda'\subset\Lambda$
(depending on $f$) such that for every $z\notin M=\mathbb{Q}\Lambda$
and $\lambda\in\Lambda'$
\[
f(z+\lambda)=f(z).
\]

If furthermore $\gcd(P,Q)=1,$ we may take $\Lambda'=\Lambda$.
\end{prop}

It remains to examine periodicity of $f$ at points $z\in M$. For
that we must assume that $P$ and $Q$ are relatively prime, as in
Theorem \ref{Periodicity Theorem}. By Lemma \ref{lem:supports} the
support of $f$ is finite modulo $\Lambda.$ Let $N$ be an integer
divisible by $PQ$ such that, with $\Lambda'=N^{-1}\Lambda$, the
function $f$ is supported on $PQ\Lambda'.$ Changing the lattice,
we are reduced to the following.
\begin{claim*}
Let $\Lambda'\subset V$ be a lattice, $N$ an integer divisible by
$PQ$ and $f:PQ\Lambda'\to\mathbb{R}$ a function. Assume that $f_{P}$
and $f_{Q}$, which are supported on $\Lambda'$, are $N\Lambda'$-periodic
for some integer $N$. Then a suitable modification of $f$ at $0$
is $N\Lambda'$-periodic.
\end{claim*}
This was proved in Proposition \ref{prop:rational_case}.

\section{A theorem on elliptic functions}

Let $\Lambda\subset\mathbb{C}$ be a lattice and $M=\mathbb{Q}\Lambda.$
Let $K$ be the field of meromorphic functions on $\mathbb{C}$ which
are periodic with respect to \emph{some} lattice $\Lambda'\subset M$.
We call such functions $M$\emph{-elliptic}. If $K_{\Lambda}$ is
the field of $\Lambda$-elliptic functions, then $K$ is the maximal
unramified extension of $K_{\Lambda}.$ 

Let $p$ and $q$ be multiplicatively independent natural numbers\footnote{For typographical reasons, we let $p$ and $q$ stand for what was
denoted $P$ and $Q$ in the previous section. The primes dividing
$P$ or $Q$ will not show up anymore.}. Consider the automorphisms
\[
\sigma f(z)=f(pz),\,\,\,\tau f(z)=f(qz)
\]
of the field $K$. Let $\widehat{K}=\mathbb{C}((z))$ and embed $K$
in $\widehat{K}$ assigning to any $f$ its Laurent series at $0$.

Let
\[
\Gamma=\left\langle \sigma,\tau\right\rangle \subset Aut(K)
\]
be the group of automorphisms of $K$ generated by $\sigma$ and $\tau.$
As $\sigma$ and $\tau$ commute, and $p$ and $q$ are multiplicatively
independent, $\Gamma\simeq\mathbb{Z}^{2}$. The group $\Gamma$ acts
of course also on $\widehat{K}$. The goal of this section is to show
how Theorem \ref{Periodicity Theorem} can be used to prove the following.
\begin{thm}
\label{thm:d=00003D1_case_conjecture}Assume that $p$ and $q$ are
relatively prime. Then the map
\[
H^{1}(\Gamma,\mathbb{C}^{\times})\to H^{1}(\Gamma,K^{\times})
\]
is an isomorphism.
\end{thm}

\begin{proof}
In this section we reserve the letter $f$ to denote elliptic functions.
Typically, if $f\in K^{\times}$,
\[
e(z)=ord_{z}(f)\in\mathscr{D}
\]
and is of course periodic.

The injectivity statement is trivial: if $f$ is $\Lambda$-elliptic
for some $\Lambda\subset M$ and $f(pz)/f(z)$ is constant then it
is easily seen that $f$ had to be constant to begin with.

For the surjectivity consider $\mathcal{D},$ the group of all the
functions $d:\mathbb{C}\to\mathbb{Z}$ with discrete support, which
are $\Lambda$-periodic for some lattice $\Lambda\subset M$. Let
$\mathcal{D}^{0}$ be the subgroup of all $d\in\mathcal{D}$ which
are of degree 0 on $\mathbb{C}/\Lambda,$ for some (equivalently,
any) lattice $\Lambda$ modulo which they are periodic. Let $\mathcal{P}\subset\mathcal{D}^{0}$
be the subgroup of principal divisors, i.e. $d$ for which there exists
a function $f\in K$ with $ord_{z}(f)=d(z),$ or $d=div(f).$ By the
Abel-Jacobi theorem a $d\in\mathcal{D}^{0}$ is principal if and only
if for some (equivalently, any) lattice $\Lambda$ modulo which $d$
is periodic, $\sum_{z\in\mathbb{C}/\Lambda}zd(z)\in M$.

Let $\{f_{\gamma}\}$ be a 1-cocycle with values in $K^{\times}$,
and choose a lattice $\Lambda$ such that $f_{\sigma}$ and $f_{\tau}$
are $\Lambda$-elliptic. From $\sigma\tau=\tau\sigma$ we get
\[
f_{\tau}(pz)/f_{\tau}(z)=f_{\sigma}(qz)/f_{\sigma}(z).
\]
If $\{d_{\gamma}\}$ is the 1-cocycle with values in $\mathcal{P}$
defined by $d_{\gamma}(z)=ord_{z}(f_{\gamma})$ then, looking at the
constant term on both sides of the last equation, we get
\[
p^{d_{\tau}(0)}=q^{d_{\sigma}(0)},
\]
hence $d_{\tau}(0)=d_{\sigma}(0)=0$. This implies that $d_{\gamma}(0)=0$
for every $\gamma\in\Gamma$. For lack of a better terminology we
call such a 1-cocycle $\{d_{\gamma}\}$ \emph{special.}

From the exactness of
\[
0\to\mathbb{C}^{\times}\to K^{\times}\to\mathcal{P}\to0
\]
we see that it is enough to prove that our special 1-cocycle $\{d_{\gamma}\}$
is a coboundary. As before, from $\sigma\tau=\tau\sigma$ we get
\begin{equation}
d_{\tau}(pz)-d_{\tau}(z)=d_{\sigma}(qz)-d_{\sigma}(z).\label{eq:cocycle-1}
\end{equation}
We have to show that there exists an $e\in\mathcal{P}$ with
\begin{equation}
d_{\sigma}(z)=e(pz)-e(z),\,\,\,d_{\tau}(z)=e(qz)-e(z).\label{eq:coboundary-1}
\end{equation}

From the equation (\ref{eq:cocycle-1}) we get
\[
d_{\tau}(z)=d_{\tau}(z/p)+d_{\sigma}(qz/p)-d_{\sigma}(z/p)=
\]
\[
d_{\tau}(z/p^{2})+d_{\sigma}(qz/p^{2})+d_{\sigma}(qz/p)-d_{\sigma}(z/p^{2})-d_{\sigma}(z/p)=\cdots
\]
\[
=\sum_{n=1}^{\infty}(d_{\sigma}(qz/p^{n})-d_{\sigma}(z/p^{n})).
\]
The sum is finite by the assumption on the supports. Thus, by telescopy,
\begin{equation}
\tilde{e}(z)=\sum_{m=1}^{\infty}d_{\tau}(z/q^{m})=\sum_{n=1}^{\infty}d_{\sigma}(z/p^{n})\label{eq:telescopy-1}
\end{equation}
satisfies (\ref{eq:coboundary-1}). Its support is discrete.

We are now in a position to apply Theorem \ref{Periodicity Theorem}.
Suitably modifying $\tilde{e}$ at $0$ we get a function $e\in\mathcal{D}$
satisfying (\ref{eq:coboundary-1}), in fact of the same periodicity
lattice $\Lambda$ of $d_{\sigma}$ and $d_{\tau}$. It remains to
show that $e\in\mathcal{P}$, i.e. that it satisfies the two conditions
prescribed by the Abel-Jacobi theorem.

Let $\Pi$ be a parllelogram which is a fundamental domain for $\mathbb{C}/\Lambda.$
Since $d_{\sigma}\in\mathcal{D}^{0}$,
\[
0=\sum_{z\in\Pi}d_{\sigma}(z)=\sum_{z\in p\Pi}e(z)-\sum_{z\in\Pi}e(z)=(p^{2}-1)\sum_{z\in\Pi}e(z),
\]
so $e\in\mathcal{D}^{0}$. Similarly
\[
\sum_{z\in\Pi}zd_{\sigma}(z)=\sum_{z\in\Pi}z(e(pz)-e(z))=p^{-1}\sum_{z\in p\Pi}ze(z)-\sum_{z\in\Pi}ze(z)=(p-1)\sum_{z\in\Pi}ze(z).
\]
Since $f_{\sigma}$ is $\Lambda$-elliptic, the left hand side lies
in $\Lambda.$ If $\Lambda'=(p-1)\Lambda$ and $\Pi'$ is a fundamental
domain for $\mathbb{C}/\Lambda'$ consisting of $(p-1)^{2}$ translates
of $\Pi$ then
\[
\sum_{z\in\Pi'}ze(z)=(p-1)^{2}\sum_{z\in\Pi}ze(z)=(p-1)\sum_{z\in\Pi}zd_{\sigma}(z)\in\Lambda'.
\]
By Abel-Jacobi, $e$ is the divisor of a $\Lambda'$-elliptic function.

We have found an $e\in\mathcal{P}$ such that $d_{\gamma}=\gamma(e)-e$
for every $\gamma\in\Gamma.$ This concludes the proof of the theorem.
\end{proof}
Let us turn to the proof of Theorem \ref{thm:criterion for ellipticity}.
Let $f$ be meromorphic in $\mathbb{C}$ and assume that
\[
f_{p}(z)=f(pz)/f(z),\,\,\,f_{q}(z)=f(qz)/f(z)
\]
are $\Lambda$-elliptic. Let
\[
d_{\sigma}(z)=ord_{z}(f_{p}),\,\,\,d_{\tau}(z)=ord_{z}(f_{q}).
\]
The relation
\[
d_{\sigma}(qz)-d_{\sigma}(z)=d_{\tau}(pz)-d_{\tau}(z)
\]
guarantees that we can extend $d$ to a special $1$-cocycle $\{d_{\gamma}\}$
of $\Gamma$ in $\mathcal{P}$. The proof of Theorem \ref{thm:d=00003D1_case_conjecture}
above yields an $e\in\mathcal{P}$ for which $d_{\gamma}=\gamma(e)-e.$
Let $\tilde{f}$ be the $\Lambda'$-elliptic function whose divisor
is $e$. Let $g=\tilde{f}/f$. Then $g(pz)/g(z)$ is periodic and
has no poles or zeros, so must be constant. This immediately implies
that $g(z)=cz^{m}$ for some $c$ and $m.$ The theorem follows.

The proof shows that $\tilde{f}$ is $\Lambda'$-periodic, where $\Lambda'=(p-1)\Lambda.$
By the same token we can take $\Lambda'=(q-1)\Lambda.$ It follows
that we can take, as the periodicity lattice of $\tilde{f},$ the
lattice $D\Lambda,$ where $D$ is the greatest common divisor of
$p-1$ and $q-1$.

\section{Higher rank analogues}

Theorem \ref{thm:d=00003D1_case_conjecture} raises a question in
non-abelian cohomology. Let $d\ge1.$ The group $\Gamma\subset Aut(K)$
acts on $GL_{d}(K)$ via its action on $K$.

\bigskip{}

\textbf{Question: }\label{conj:Main Conjecture}Assume that $p$ and
$q$ are multiplicatively independent and $d\ge1.$ Is the map of
pointed sets
\[
H^{1}(\Gamma,GL_{d}(\mathbb{C}))\to H^{1}(\Gamma,GL_{d}(K))
\]
bijective? If not, is it injective? Can we identify its image?\bigskip{}

When $K=\bigcup\mathbb{C}(z^{1/n}),$ $\sigma(f)(z)=f(z^{p})$ and
$\tau(f)(z)=f(z^{q}),$ the analogous map is bijective. This is due
entirely to Schäfke and Singer, even if \cite{7} falls short of formulating
it in cohomological terms. See also \cite{4}.

In \cite{8} we show that the answer to the above question is \emph{negative}
as soon as $d\ge2.$ The\emph{ reason} for the different behavior
in the case of $\mathbb{G}_{m}=\mathbb{P}^{1}-\{0,\infty\},$ the
algebraic group underlying the rational case studied in \cite{7},
and the elliptic case, turns out to be that while every vector bundle
on $\mathbb{G}_{m}$ is trivial, there are non-trivial vector bundles
on elliptic curves which are \emph{invariant under pull-back by all
isogenies}. These vector bundles have been classified by Atiyah in
1957, and sometimes bear his name.

In \cite{8} we prove a vast generalization of the periodicity criterion
proved in Theorem \ref{Periodicity Theorem}. Using it, we associate
to a given class in $H^{1}(\Gamma,GL_{d}(K))$ a vector bundle on
$\mathbb{C}/\Lambda$, for all small enough $\Lambda.$ It turns out
that the map $H^{1}(\Gamma,GL_{d}(\mathbb{C}))\to H^{1}(\Gamma,GL_{d}(K))$
is injective, and its image consists of the classes whose associated
vector bundle is trivial.

We end by giving an example of a cohomolgy class in $H^{1}(\Gamma,GL_{2}(K))$
that does not come from a similar class over $\mathbb{C}$ by base
change. Let $\Lambda\subset\mathbb{C}$ be a lattice and let $\zeta(z)=\zeta(z,\Lambda)$
be the Weierstrass zeta function of $\Lambda.$ Recall that
\[
\zeta'(z,\Lambda)=\wp(z,\Lambda)
\]
is the Weierstrass $\wp$-function, but for $0\ne\omega\in\Lambda$
\[
\zeta(z+\omega)-\zeta(z)=\eta(\omega,\Lambda)
\]
is a non-zero constant. Let
\[
\begin{cases}
\begin{array}{c}
g_{p}(z)=p\zeta(qz)-\zeta(pqz)\\
g_{q}(z)=q\zeta(pz)-\zeta(pqz)
\end{array}.\end{cases}
\]
Clearly, $g_{p},g_{q}$ are $\Lambda$-elliptic functions. Let
\[
A=\left(\begin{array}{cc}
1 & g_{p}(z)\\
0 & p
\end{array}\right),\,\,\,B=\left(\begin{array}{cc}
1 & g_{q}(z)\\
0 & q
\end{array}\right).
\]
It can be checked that there is a cocycle of $\Gamma$ in $GL_{2}(K)$
sending $\sigma^{-1}$ to $A$ and $\tau^{-1}$ to $B$. Since $\Gamma$
is free abelian, this amounts to checking the \emph{consistency equation}
\[
A(z/q)B(z)=B(z/p)A(z)
\]
 which the reader may easily verify. 

In \cite{8} we show that this cocycle represents a cohomology class
that does not arise form a similar class over $\mathbb{C}.$ In the
language of difference equations, the pair $(A,B)$ is not guage-equivalent
to a pair $(A_{0},B_{0})$ of scalar matrices. In fact, the results
of \cite{8} show that every class in $H^{1}(\Gamma,GL_{2}(K))$ that
is not in the image of $H^{1}(\Gamma,GL_{2}(\mathbb{C}))$ is represented
by a pair of matrices $(aA,bB)$ with $A,B$ as above and $a,b\in\mathbb{C}^{\times}.$
Similar, but more complicated, results hold in higher ranks.

\end{document}